\documentclass[11pt]{amsart}
\newtheorem{theoreme}{Theorem}[subsection]
\newtheorem{fait}[theoreme]{Fact}

\newtheorem{lemme}[theoreme]{Lemma}
\newtheorem{proposition}[theoreme]{Proposition}

\newtheorem{question}[theoreme]{Question}
\newtheorem{conjecture}[theoreme]{Conjecture}
\input xy
\xyoption{all}
\usepackage{amsmath}
\usepackage{amssymb}
\usepackage{amscd}
\usepackage{latexsym}
\usepackage[latin1]{inputenc}
\usepackage{enumerate}
\usepackage[all]{xy}
\usepackage{hyperref}

\usepackage{graphicx}
\usepackage{changebar,color}
\usepackage{fullpage}
\usepackage{enumerate}
\usepackage[a4paper, left=2cm,right=1.5cm,bottom=2cm,top=2cm]{geometry}
\usepackage{color}
\usepackage{graphicx,epsfig}
\usepackage{amscd}
\usepackage{amssymb}
\usepackage{enumitem}
\usepackage{ulem}
\usepackage{hyperref}
\hypersetup{colorlinks,allcolors=blue}

\newcommand{\Q}{\mathbb{Q}}
\newcommand{\Z}{\mathbb{Z}}

\newcommand{\SGL}{{\hbox{\rm GL}}}

\newcommand{\SH}{{\hbox{\rm H}}}

\newcommand{\Sim}{\hbox{\rm im}}

\newcommand{\ra}{\rightarrow}
\newcommand{\C}{\mathbb C}

\title{A remark on uniform boundedness for Brauer groups}

\author{Anna Cadoret and Fran\c{c}ois Charles}

\date{}
\begin{document}
\maketitle
 \textit{}\\
\begin{abstract}
\noindent \scriptsize{ }  \\
The Tate conjecture for divisors on varieties over number fields is equivalent to finiteness of $\ell$-primary torsion in the Brauer group. We show that this finiteness is actually uniform in one-dimensional families for varieties that satisfy the Tate conjecture for divisors -- e.g. abelian varieties and $K3$ surfaces.

 \end{abstract}\textit{} \\
 \normalsize

 \section{Introduction}

\subsection{}\hspace{-0.3cm} Let $k$ be a field of characteristic $p\geq 0$ and let $\overline k$ be a separable closure of $k$. Let $X$ be a smooth proper variety over  $k$. For any prime number $\ell\not=p$, consider the cycle class map
$$c_1 : \mathrm{Pic}(X_{\overline k})\otimes\Q_\ell\ra H^2(X_{\overline k}, \Q_\ell(1)).$$
The absolute Galois group  $Gal(\overline k/k)$ of $k$ acts naturally on the \'etale cohomology group $H^2(X_{\overline k}, \Q_\ell(1))$, and the image of the cycle class map $c_1$ is contained in the group 
$$\bigcup_{U} H^2(X_{\overline k}, \Q_\ell(1))^U,$$
where $U$ runs through the open subgroups of $Gal(\overline k/k)$. We say that $X$ satisfies the $\ell$-adic Tate conjecture for divisors if the image of $c_1$ is equal to the group above.

As is well-known, the Tate conjecture for divisors is related to finiteness results for the Brauer group. Indeed, given $X$ and $\ell$ as above, $X$ satisfies the $\ell$-adic Tate conjecture for divisors if and only if for every subgroup $U$ of finite index in $Gal(\overline k/k)$, the group 
$$Br(X_{\overline k})^U[\ell^{\infty}]$$ 
is finite.


\subsection{}\hspace{-0.3cm} In this paper, we investigate whether one might be able to find uniform bounds on the Brauer groups above when the variety $X$ varies in a family. Following V\'arilly-Alvarado's conjectures on $K3$ surfaces \cite[par. 4]{Varilly17}, such a question might be regarded as an analogue of results of Manin \cite{Manin}, Faltings-Frey \cite{FreyF}, Mazur \cite{Mazur} and Merel \cite{Merel} on uniform bounds for torsion of elliptic curves. Related boundedness results for families of smooth proper varieties have been developed by Cadoret-Tamagawa \cite{UOI1}, \cite{UOI2}, \cite{TI}, \cite{GA}. It should be noted that all the aforementioned results can only treat the case of one-dimensional bases. Dealing with the higher-dimensional case seems to be much harder, as rational points are not well-understood on varieties of dimension at least $2$.

The main result of this paper is as follows.

\begin{theoreme}\label{theorem:main} Let $k$ be a finitely generated field of characteristic zero. Let $S$ be a curve over $k$, and let $\pi : X\ra S$ be a smooth, proper morphism. Let $\ell$ be a prime number. Then for every positive integer $d$,  there exists a constant $C_d$  satisfying the following property: for every field extension $k\subset K\subset \overline{k}$ with degree $[K:k]\leq d$ and every $K$-point $s$ such that  $X_s$ satisfies the $\ell$-adic Tate conjecture for divisors $$\big|Br(X_{s,\overline{k}})^{Gal(\overline k/K)}[\ell^\infty]\big|\leq C_d.$$
\end{theoreme}

Examples of projective varieties for which the Tate conjecture for divisors is known -- and so for which Theorem \ref{theorem:main} applies  -- include abelian varieties and $K3$ surfaces \cite{Faltings83, TateMotives}.

In the case of certain one-parameter families of abelian and $K3$ surfaces -- which are actually parametrized by Shimura curves -- the theorem above was proved by V\'arilly-Alvarado-Viray \cite[Theorem 1.5, Corollary 1.6]{VAV}. While their proof relies on deep arithmetic properties of elliptic curves, Theorem \ref{theorem:main} is derived as a consequence of the results of Cadoret-Tamagawa in \cite{UOI2}. We will also prove a version of Theorem \ref{theorem:main} that does not involve the Tate conjecture -- see Theorem \ref{theorem:mainbis}.

\bigskip

\noindent\textbf{Outline of the paper.} Section \ref{section:generalities} is devoted to some general statements and questions about Brauer groups and their behavior in families. In Section \ref{section:proof-main} we prove Theorem \ref{theorem:main} and its unconditional variant.
\bigskip

\noindent\textbf{Acknowledgements.} The work on this note started after listening to an inspiring talk of V\'arilly-Alvarado on his joint work with Viray at the conference "Arithmetic Algebraic Geometry" at the Courant Institute in september 2016. We thank Tony V\'arilly-Alvarado for useful discussions and sharing his conjectures with us. We thank Emiliano Ambrosi for helpful discussions. \\

This project has received funding from  the ANR grant ANR-15-CE40-0002-01 and, for the  second author, from the European Research Council (ERC) under the European Union's Horizon 2020 research and innovation programme (grant agreement No 715747).  

\bigskip

\noindent\textbf{Notation.} If $A$ is an abelian group, and $n$ is an integer, we write $A[n]$ for the $n$-torsion subgroup of $A$. If $\ell$ is a prime number, we write 
$$A[\ell^\infty]:=\lim_{\ra} A[\ell^n]$$
for the $\ell$-primary torsion subgroup of $A$, and 
$$T_\ell(A):=\lim_{\leftarrow} A[\ell^n]$$
for its $\ell$-adic Tate module.

Let $X$ be a smooth proper variety over a field $k$. We write $Br(X)$ for the Brauer group of $X$, that is, the \'etale cohomology group $H^2(X, \mathbb G_m)$, and $NS(X)$ for the N\'eron-Severi group of $X$, that is, the image of $Pic(X)$ in $H^2(X_{\overline k}, \Q_{\ell}(1))$, where $\overline k$ is a separable closure of $k$.

\section{A question on uniform finiteness for Brauer groups}\label{section:generalities}

In this section, we recall the relationship between the Tate conjecture and finiteness of the Brauer group. While the Tate conjecture is a qualitative statement, finiteness problems can be made quantitative, which leads us to Question \ref{conjecture:boundedness-in-families}.

\subsection{}\hspace{-0.3cm}The following proposition provides the basic relationship between the $\ell$-adic Tate conjecture and finiteness for Brauer groups.

Let $k$ be a field of characteristic $p\geq 0$. Let $X$ be a smooth proper variety over $k$. Choose $\ell$ a prime different from $p$. Let $B:=Br(X_{\overline{k}})$ be the Brauer group of $X_{\overline k}$, and set
$$T_\ell:=T_\ell(B),\, V_\ell:=T_\ell\otimes\Q_\ell, \,M_\ell:=T_\ell\otimes\Q_\ell/\Z_\ell\simeq V_\ell/T_\ell.$$
For every integer $n\geq 0$, we can identify $M_\ell[\ell^n]$ with $T/\ell^n$.\\

\begin{proposition}\label{proposition:basic-bound}
The following assertions are equivalent
 \begin{enumerate}
 \item The $\ell$-adic Tate conjecture holds for divisors on $X$;
 \item  $B^U[\ell^\infty]$ is finite for every open subgroup $U\subset  Gal(\overline k/k)$;
 \item  $M_\ell^U $ is finite for every open subgroup $U\subset  Gal(\overline k/k)$;
 \item  $T_\ell^U=0$ for every open subgroup $U\subset  Gal(\overline k/k)$.

  \end{enumerate}
And,  if this holds, for every open subgroup $U\subset Gal(\overline k/k)$ one has 
\begin{enumerate}
\setcounter{enumi}{4}
\item  $\big|B^U[\ell^\infty]\big|\leq \big|M_\ell^U\big|+\big| \SH^3(X_{\overline{k}},\Z_\ell(1))[\ell^\infty]\big|.$
\end{enumerate}

 \end{proposition}
 
 \begin{proof} We first prove the equivalence of  $(2), (3)$ and $(4)$.  We can restrict to the case where $U$ is normal. More precisely, we show that for any given normal subgroup $U\subset   Gal(\overline k/k)$ the following are equivalent
 \begin{enumerate} 
 \setcounter{enumi}{1}
 \item  $B^{U}[\ell^\infty]$ is finite ;
 \item   $M_\ell^U $ is finite;
 \item   $T_\ell^{U}=0$,
 \end{enumerate}
 and, that, if this holds, one has  
\begin{enumerate}
\setcounter{enumi}{4}
\item  $\big|B^U[\ell^\infty]\big|\leq \big|M_\ell^U\big|+\big| \SH^3(X_{\overline{k}},\Z_\ell(1))[\ell^\infty]\big|.$
\end{enumerate}

For every positive integer $n$, the cohomology of the Kummer short exact sequence 
$$1\ra\mu_{\ell^n}\ra\mathbb{G}_m\ra\mathbb G_m\ra 1$$
induces an exact sequence 
$$ 0\ra NS(X_{\overline{k}})\otimes\Z/\ell^n\ra  \SH^2(X_{\overline{k}},\mu_{\ell^n})\ra B[\ell^n]\ra  0.$$
 Letting $n$ tend to $\infty$, we get a commutative diagram with exact lines and columns
 $$\xymatrix{ & 0\ar[d]& 0\ar[d]& 0\ar[d]&  \\
 0\ar[r]& \ell^n NS(X_{\overline{k}})\otimes\Z_\ell\ar[r]\ar[d]&  \ell^n \SH^2(X_{\overline{k}},\Z_\ell(1))\ar[r]\ar[d]&  K_\ell \ar[d]&   \\
 0\ar[r]& NS(X_{\overline{k}})\otimes\Z_\ell\ar[r]\ar[d]&  \SH^2(X_{\overline{k}},\Z_\ell(1))\ar[r]\ar[d]&  T_\ell\ar[r]\ar[d]&  0\\
 0\ar[r]& NS(X_{\overline{k}})\otimes\Z/\ell^n\ar[r]\ar[d]&  \SH^2(X_{\overline{k}},\mu_{\ell^n})\ar[r]\ar[d]&   B[\ell^n]\ar[r]\ar[d]&  0\\
 &0\ar[r]& \SH^3(X_{\overline{k}},\Z_\ell(1))[\ell^n]\ar[r]\ar[d]&C\ar[r]\ar[d]&0\\
 &&0 &0 & }$$
 whence a short exact sequence 
 $$0\rightarrow  T_\ell\otimes \Z/\ell^n\rightarrow B[\ell^n] \rightarrow  \SH^3(X_{\overline{k}},\Z_\ell(1))[\ell^n]\rightarrow 0.$$
 The exactness on the left follows from the snake, which ensures that $K_\ell=\ell^nT_\ell$.

Taking inductive limits and identifying $T_\ell\otimes \Z/\ell^n\simeq M[\ell^n]$, we obtain a short exact sequence 
$$
 0\rightarrow  M_\ell \rightarrow B[\ell^\infty] \rightarrow  \SH^3(X_{\overline{k}},\Z_\ell(1))[\ell^\infty]\rightarrow 0$$
and taking $U$-invariants, we obtain an exact sequence 
$$
(\ref{proposition:basic-bound}.6) \;\; 0\rightarrow  M_\ell^{U} \rightarrow B^U[\ell^\infty] \rightarrow  \SH^3(X_{\overline{k}},\Z_\ell(1))^U[\ell^\infty].
$$

This shows (2) $\Leftrightarrow$ (3) and that, in that case, (5) holds. For the equivalence  (3) $\Leftrightarrow$ (4), let $U_\ell\subset \SGL(T_\ell)$ denote the image of $U\subset Gal(\overline k/k)$ acting on $T_\ell$ and consider the exact sequence 
$$0\rightarrow T_\ell^U\rightarrow V_\ell^U\rightarrow M_\ell^U\stackrel{\delta}{\rightarrow} \SH^1(U_\ell,T_\ell).$$
Since $U_\ell$ is a compact $\ell$-adic Lie group, $ \SH^1(U_\ell,T_\ell)$ is a  topologically finitely generated $\Z_\ell$-module \cite[Prop. 9]{GCVA}. In particular, $\Sim(\delta)$ is finite. The conclusion follows from the fact that either $T_\ell^U=V_\ell^U=0$ 
or $V_\ell^U/T_\ell^U$ is infinite. 

Now, assume $X$ satisfies the $\ell$-adic Tate conjecture for divisors. Then, since $NS(X_{\overline k})$ is a finitely generated abelian group (\cite[p. 145, Thm. 2]{Neron}, \cite[XIII, 5.1]{SGA6}), we can find an open subgroup $U_0$ of $Gal(\overline k/k)$ such that the action of $U_0$ on $NS(X_{\overline k})$ is trivial, and the map 
 $$NS(X_{\overline k})\otimes\Q_\ell\ra \SH^2(X_{\overline{k}},\Q_\ell(1))^{ U}$$ 
 is onto for every open subgroup $U$ of $U_0$. Furthermore, \cite[Prop. (5.1)]{TateMotives} shows that as a consequence of the Tate conjecture for $X$, the exact sequence of $U$-modules
 $$0\ra NS(X_{\overline k})\otimes\Q_\ell\ra \SH^2(X_{\overline{k}},\Q_\ell(1))\ra V_\ell\ra 0$$
 is split, so that the surjectivity above implies $V_\ell^U=0.$ Since $T_\ell$ is a torsion-free $\Z_\ell$-module, this implies $T_\ell^U=0.$ Since this vanishing holds for any small enough open subgroup $U$ of $Gal(\overline k/k)$, it holds for any open subgroup $U$ of $Gal(\overline k/k)$. This shows  (1) $\Rightarrow$ (4). For the converse implication,  let $U_0\subset Gal(\overline k/k)$ be an open subgroup such that $$H^2(X_{\overline k}, \Q_\ell(1))^{U_0}=
 \bigcup_{[ Gal(\overline k/k):U]<\infty} H^2(X_{\overline k}, \Q_\ell(1))^U.$$
 Since $T_\ell^{U_0}=0$ implies $V_\ell^{U_0}=0$, we obtain an isomorphism  $$  (NS(X_{\overline{k}})\otimes\Q_\ell)^{U_0}\tilde{\rightarrow}H^2(X_{\overline k}, \Q_\ell(1))^{U_0} $$ hence, \textit{a fortiori},  $   NS(X_{\overline{k}})\otimes\Q_\ell\tilde{\rightarrow}H^2(X_{\overline k}, \Q_\ell(1))^{U_0}.$ 
 \end{proof}

\subsection{}\hspace{-0.3cm} With Proposition \ref{proposition:basic-bound} in mind, we offer the following question on uniform boundedness for Brauer groups, which we view as an enhancement of the $\ell$-adic Tate conjecture for divisors. Theorem \ref{theorem:main} settles the case where $S$ is one-dimensional for those points that satisfy the Tate conjecture.

\begin{question}\label{conjecture:boundedness-in-families}
Let $k$ be a finitely generated field of characteristic zero, and let $\ell$ be a prime. Let $S$ be a quasi-projective variety over $k$. Let $\pi : X\ra S$ be a smooth projective morphism. Then for any positive integer $d$, there exists a constant $C_d$ satisfying the following property: for any field extension $k\subset K\subset\overline k$ of degree $[K:k]\leq d$, and any $K$-point $s\in  S(K)$, we have 
$$|Br(X_{s, \overline k})^{Gal(\overline{k}/k)}[\ell^\infty]|\leq C_d.$$
\end{question}

For $d=1$, Question \ref{conjecture:boundedness-in-families} is closely related to (the   $\ell$-adic Tate conjecture for divisors and) the Bombieri-Lang conjecture (see e.g. \cite[Conj. 1.8]{Brunebarbe}). More precisely, assuming that the $X_s$, $s\in S(k)$ satisfies the  $\ell$-adic Tate conjecture for divisors, Question \ref{conjecture:boundedness-in-families} would follow from the following Galois-theoretic conjecture. For every integer $C\geq 1$ let $S^{\leq C}(k)$ be the set of all $s\in S(k)$ such that  the  image of $Gal(\overline{k}/k)$ acting on  $H^2(X_{\overline{s}},\Q_\ell)$ is  open of index at most $ C$ in the image of $\pi_1(S)$ and let $S^{\leq C, ex}(k)$ be the complement of $S^{\leq C}(k)$ in $S(k)$.  Then we have the following:

\begin{conjecture}\label{conjecture:Galois} There exists an integer $C\geq 1 $ such that  $S^{\leq C, ex}(k)$ is not Zariski-dense in $S$. 
\end{conjecture}

\noindent When $S$ is a curve, Conjecture \ref{conjecture:Galois} is a special case of the main result of \cite{UOI1}.  We explain the relation between the general case of Conjecture \ref{conjecture:Galois}  and the Bombieri-Lang conjecture. One can always construct a strictly increasing  sequence $C_n$, $n\geq 0$ of integers and a projective system 
$$\ldots\ra S_{n+1}\rightarrow S_{n}\ra S_{n-1}\ra\ldots$$
of \'etale covers of $S$ with the property that the image of $S_{n}(k)$ in $S$ is   $S^{ \leq C_n, ex}(k)$ (see \cite[\S 3.1.2]{UOI1}). So proving that $ S^{ \leq C, ex}(k)$ is not Zariski-dense in $S$ for $C$ large enough amounts  to proving that $S_{n}(k)$ is not Zariski-dense in $S_{n}$ for $n\gg 0$. According to the Bombieri-Lang conjecture, this would follow from the fact that  $S_{n}$ is of general type for $n\gg 0$. While the Bombieri-Lang conjecture is still completely open, proving that $S_{n}$ is of general type for $n\gg 0$ seems to be a more tractable problem. The projective system $(S_n)$ can indeed be regarded as a generalization of towers of Shimura varieties. In this setting, Brunebarbe recently obtained an asymptotic hyperbolicity result (implying the general type property)  -- see \cite{Brunebarbe} and the references therein, as well as \cite{AVA}. 



\bigskip

As a concrete example, we now explain how a positive answer to the question above in the case of families of $K3$ surfaces -- for which the Tate conjecture is known -- yields a positive answer to the $\ell$-primary version of Conjecture 4.6 of V\'arilly-Alvarado in \cite{Varilly17}.

\begin{proposition}\label{conjecture:Varilly}
Assume that Question \ref{conjecture:boundedness-in-families} has a positive answer for families of $K3$ surfaces. Let $d$ be a positive integer, and let $\Lambda$ be a lattice. Then there exists a positive constant $C_d$ such that for any degree $d$ number field $k$ with algebraic closure $\overline k$, and any $K3$ surface $X$ over $k$ with $NS(X_{\overline k})\simeq \Lambda$, the following inequality holds:
$$|Br(X_{\overline k})^{Gal(\overline k/k)}[\ell^\infty]|\leq C_d.$$
\end{proposition}

Using theorem \ref{theorem:main}, the proposition above is a consequence of the following result.

\begin{proposition}\label{proposition:boundedness}
There exists a positive integer $N$ with the following property: let $\Lambda$ be a lattice. Then there exists a smooth projective family of $K3$ surfaces $\pi : \mathcal X\ra S$ over a quasi-projective base $S$ over $\Q$ such that given any field $k$ of characteristic zero with algebraic closure $\overline k$, and any $K3$ surface $X$ over $k$ with $NS(X_{\overline k})\simeq \Lambda$, there exists an extension $K/k$ of degree $N$, and a $K$-point $s$ of $S$ with $\mathcal X_s\simeq X_K$. 
\end{proposition}

We start with a lemma. If $L$ is a line bundle on a $K3$ surface $X$, write $c_1(L)$ for its class in $NS(X)\simeq \mathrm{Pic}(X)$.

\begin{lemme}\label{lemma:bounded-degree}
There exists a positive integer $N$ with the following property: let $\Lambda$ be a lattice. There exists a positive integer $\delta_\Lambda$, depending only on $\Lambda$, such that for any field $k$ with algebraic closure $\overline k$, and any $K3$ surface $X$ over $k$ with $NS(X_{\overline{k}})\simeq \Lambda$, there exists an extension $K/k$ of degree $N$ and an ample line bundle $L$ on $X_K$ with $c_1(L)^2=\delta_\Lambda.$
\end{lemme}

\begin{proof}
Given any integer $n$, there exist only finitely many finite groups that can occur as subgroups of $GL_n(\Z)$. Since the continuous action of any profinite group on a discrete finitely generated abelian group factors through a finite quotient, and since the Néron-Severi group of any $K3$ surface is free of rank at most $22$, this shows that there exist finitely many finite groups that can occur as the image of $Gal(\overline k/k)\ra NS(X_{\overline k})$ for a $K3$ surface $X$ over a field $k$ with algebraic closure $\overline k$. 

In particular, there exists an integer $N$, independent of $X$ and $k$, with the following property: for any $k, \overline k$ and $X$ as above, there exists a Galois extension $k\subset K\subset \overline k$ of degree at most $N$ such that the action of $Gal(\overline k/k)$ on $NS(X_{\overline k})$ factors through $Gal(K/k)$.

\bigskip

Let $X$ be a $K3$ surface over a field $k$ with algebraic closure $k$, such that $NS(X_{\overline{k}})\simeq \Lambda$. By \cite[Lemma 2.3.2]{LMS14}, there exists a constant $C_\Lambda$, depending only on $\Lambda$, such that there exists an ample line bundle $\overline L$ on $X_{\overline k}$ with $c_1(\overline L)^2=C_\Lambda$. Write $\alpha=c_1(\overline L)$. Then, $N$ being as above, we can find a Galois extension $K$ of $k$ of degree $N$ such that $\alpha$ is $Gal(\overline{k}/K)$-invariant. 

Consider the exact sequence  coming from the Hochschild-Serre spectral sequence and Hilbert 90
$$0\ra NS(X_K)\ra NS(X_{\overline{k}})^{Gal(\overline{k}/K)}\ra Br(K)\ra Br(X).$$
The zero-cycle $c_2(X)$ on $X$ has degree $24$, and induces a morphism $Br(X)\ra Br(K)$ such that the composition $Br(K)\ra Br(X)\ra Br(K)$ is multiplication by $24$. This shows that the cokernel of the map $NS(X)\ra NS(X_{\overline{k}})^{Gal(\overline{k}/K)}$ is killed by $24$. As a consequence, we can find a line bundle $L$ on $X_K$ with $c_1(L)=24\alpha$. 

The pullback of $L$ to $X_{\overline k}$ is $\overline{L}^{\otimes 24}$, which is ample. Since ampleness is a geometric property, $L$ itself is ample. We have $c_1(L)^2=24^2C_\Lambda=:\delta_\Lambda$, which only depends on $\Lambda$.
\end{proof}

\begin{proof}[Proof of Proposition \ref{proposition:boundedness}] Apply the previous lemma to find an integer $N$ such that 
for $X$ as in the proposition, there exists an extension $K$ of $k$ of degree $N$, and an ample line bundle $L$ on $X_K$ with $c_1(L)^2=\delta $, where   $\delta $ only depends on $\Lambda$.

By Koll\'ar-Matsusaka's refinement of Matsusaka's big theorem \cite{KM}, and since the canonical bundle of a $K3$ surface is trivial by definition, we can find integers $i$ and $r$, depending only on $\Lambda$, such that $L^{\otimes i}$ is very ample, and induces an embedding of $X$ into a projective space of dimension at most $r$ as a subvariety of degree at most $\delta  i^2$. The existence of the Hilbert scheme (or the theory of Chow forms) allows us to conclude.
\end{proof}

\section{Proof of Theorem \ref{theorem:main} and its unconditional variant} \label{section:proof-main} 

Let $k$ be a finitely generated field of characteristic $0$.

\subsection{}\hspace{-0.3cm}Let $S$ be a smooth, separated and geometrically connected variety over $k$, with generic point $\eta$. If $s$ is a point of $S$, we will always write $\overline s$ for a geometric point lying above $s$.

\bigskip

Let $\pi : X\ra S$ be a smooth, proper morphism.   For every prime $\ell$, let $$\rho_\ell:\pi_1(S)\rightarrow \SGL(\SH^2(X_{\overline{\eta}},\Q_\ell(1)))$$
denote the corresponding representation\footnote{The choice of base points for \'etale fundamental groups will play no part in the sequel; we ignore such choices systematically in the following and do not mention base points in the notation.}. Let $K$ be a field, and let $s$ be a $K$-point of $S$. We write $\pi(s)$ for the absolute Galois group of $K$ considered with its natural map to $\pi_1(S)$, and we write also $\rho_\ell$ for the representation of $\pi_1(s)$ obtained by composing the above with $\pi_1(s)\ra \pi_1(S)$.

Let $S_{\ell}^{gen}$ denote the set of all points $s$ of $S$ with value in a field such that $\rho_\ell(\pi_1(s))$ is an open subgroup of $\rho_\ell(\pi_1(S))$. Let $S_\ell^{ex}$ be the complement of $S_\ell^{gen}$. In other words, $S_\ell^{gen}$ is the set of points $s$ such that the Zariski closures of $\rho_\ell(\pi_1(s))$ and $\rho_\ell(\pi_1(S))$ have the same neutral component. Note that the generic point $\eta$ always belongs to $S_\ell^{gen}.$

\subsection{}\hspace{-0.3cm}We start with the following proposition, which will allow us to compare the Galois-invariant part of Brauer groups of points in $S_{\ell}^{gen}$.

 \begin{proposition}\label{proposition:lifting-Tate}
For every $s\in S_{\ell}^{gen}$,   the natural specialization maps
$NS(X_{\overline\eta})\ra NS(X_{\overline s})$
and
$Br(X_{\overline\eta})\ra Br(X_{\overline s})$
are isomorphisms.   In particular, if $X_s$ satisfies the $\ell$-adic Tate conjecture for divisors then $X_{s'}$ satisfies the $\ell$-adic Tate conjecture for divisors for all $s'\in S_{\ell}^{gen}$.   
\end{proposition}

\begin{proof}
For every prime $\ell'$, consider the commutative specialization diagram, with exact rows
$$(\ref{proposition:lifting-Tate}.1)\;\; \xymatrix{
0\ar[r] & NS(X_{\overline\eta})\otimes\Z_{\ell'}\ar[r]\ar[d]&  \SH^2(X_{\overline\eta},\Z_{\ell'}(1))\ar[d]\ar[r] & T_{\ell'}(Br(X_{\overline\eta}))\ar[d]\ar[r] & 0\\
0\ar[r] & NS(X_{\overline{s}})\otimes\Z_{\ell'}\ar[r] &  \SH^2(X_{\overline{s}},\Z_{\ell'}(1))\ar[r]& T_{\ell'}(Br(X_{\overline{s}}))\ar[r] & 0
.}$$
By the smooth base change theorem, the middle vertical map is an isomorphism, so that the leftmost vertical map is injective, and the rightmost vertical map is surjective. Since this holds for every $\ell'$, the specialization map $NS(X_{\overline\eta})\ra NS(X_{\overline s})$ is injective.

 Let $U$ be an open subgroup of $\rho_\ell(\pi_1(s))$ acting trivially on $NS(X_{\overline{s}})$, and such that $NS(X_{\overline{s}})\otimes\Z_\ell$ injects into $\SH^2(X_{\overline{s}},\Z_\ell(1))^U$. Since $s\in S_\ell^{gen}$, $U$ is open in $\rho_\ell(\pi_1(S))$ hence (\ref{proposition:lifting-Tate}.2) $U\cap \rho_\ell(\pi_1(S_{\overline{k}}))$ is  open in $\rho_\ell(\pi_1(S_{\overline{k}}))$.

Choose an embedding $\overline k\subset\C$. We may assume  $\overline s\in S(\C)$. Let   $[L]\in NS(X_{\overline{s}})$, where $L$ is a line bundle on $X_{\overline{s}}$. From (\ref{proposition:lifting-Tate}.2) and comparison between the analytic and \'etale sites, the image $c_1(L)$ of $[L]$ in $H^2(X_{\overline{s}}, \Q(1))$ is invariant under a finite index subgroup   of $\pi_1(S_{\C})$ corresponding to a finite \'etale cover $S'_\C$ of $S_\C$. Fix a smooth compactification $\overline{X}_{S'_\C}$ of $X_{S'_\C}$. The theorem of the fixed part and the degeneration of the Leray spectral sequence show that $H^2(X_{\overline{s}}, \Q(1))^{\pi_1(S'_\C)}$ is a sub-Hodge structure of $H^2(X_{\overline{s}}, \Q(1))$ onto which the Hodge structure $H^2(\overline{X}_{S'_\C}, \Q(1))$ surjects. Semisimplicity of the category of polarized Hodge structures ensures that $c_1(L)$ can be lifted to a Hodge class in $H^2(\overline{X}_{S'_\C}, \Q(1))$. By the Lefschetz $(1,1)$ theorem, this implies that some multiple of $[L]$ can be lifted to an element of $NS(X_{\overline\eta})$.
As a consequence, the specialization map
$$NS(X_{\overline\eta})\ra NS(X_{\overline s})$$
has torsion cokernel. Since for any prime number $\ell'$, both cycle class maps to $\ell'$-adic cohomology have torsion-free cokernel -- as the Tate modules are torsion-free -- the specialization map is surjective, so it is an isomorphism. 

Finally, for every prime  $\ell'$ and  positive integer $n$, consider the specialization diagram 
$$\xymatrix{
0\ar[r] & NS(X_{\overline\eta})\otimes\Z/\ell'^n\Z\ar[r]\ar[d]&  \SH^2(X_{\overline\eta},\mu_{\ell'^n})\ar[d]\ar[r] & Br(X_{\overline\eta})[\ell's^n]\ar[d]\ar[r] & 0\\
0\ar[r] & NS(X_{\overline{s}})\otimes\Z/\ell'^n\Z\ar[r] &  \SH^2(X_{\overline{s}},\mu_{\ell'^n})\ar[r]& Br(X_{\overline{s}})[\ell'^n]\ar[r] & 0
.}$$
Both the leftmost and the middle vertical maps are isomorphisms, so the rightmost one is an isomorphism as well. Since the Brauer groups are torsion, this proves that the specialization map on Brauer groups is an isomorphism.  This concludes the proof of the first part of the assertion in Lemma \ref{proposition:lifting-Tate}. The second part of the assertion is a formal consequence of the first.
\end{proof}

\subsection{}\hspace{-0.3cm}For every open normal subgroup $U\subset \rho_\ell(\pi_1(S))$, let $S_U\subset S$ denote the set of all points $s$ with value in a field such that $\rho_\ell(\pi_1(s))\supset U$.  The set $S_U$  is well-defined since $\rho_\ell(\pi_1(s))$ is well-defined up to conjugation by an element of $\rho_\ell(\pi_1(S))$ and $U$ is normal in $\rho_\ell(\pi_1(S))$.  By definition,  $S_U\subset S_{\ell}^{gen}$ and $S_{\ell}^{gen}$ is the union of the $S_U$ as $U$ runs through all normal open subgroups of $\rho_\ell(\pi_1(S))$.\\

\begin{proposition}\label{proposition:bound-fixed-U}
Let $U\subset \rho_\ell(\pi_1(S))$ be a normal open subgroup.
\begin{enumerate}
\item If there exists $s_0\in S_{\ell}^{gen}$ such that the variety $X_{s_0}$ satisfies the $\ell$-adic Tate conjecture for divisors, then for every $s\in S_U$, there exists a constant $C_U\geq 1$ such that
$$|Br(X_{\overline s})^{\pi_1(s)}[\ell^\infty]|\leq C_U.$$
 \item  Assume that the Zariski-closure $G_\ell$ of $\rho_\ell(\pi_1(S))$ is connected.  Then there exists a constant $C_U\geq 1$ such that for every $s\in S_U$
$$[Br(X_{\overline s})^{\pi_1(s)}[\ell^\infty]:Br(X_{\overline\eta})^{\pi_1(S)}[\ell^\infty]]\leq C_U.$$
\end{enumerate}
\end{proposition}

\begin{proof} Write $B=Br(X_{\overline\eta})$, $T_\ell=T_\ell(B)$, $V_\ell=T_\ell\otimes\Q_\ell$ and $M_\ell=V_\ell/T_\ell$.\\

\noindent By the second part of Proposition \ref{proposition:lifting-Tate} and Proposition \ref{proposition:basic-bound}, $B^U[\ell^\infty]$ is finite. By the first part of Proposition \ref{proposition:lifting-Tate}, the specialization map $B\ra Br(X_{\overline s})$ is an isomorphism, which is obviously Galois-equivariant. In particular, since $\rho_\ell(\pi_1(s))$ contains $U$, we have a natural injection
$$Br(X_{\overline s})^{\pi_1(s)}[\ell^\infty]\hookrightarrow B^U[\ell^\infty].$$
This shows (1). \\

\noindent Let $s\in S_U$. From the exact sequence (\ref{proposition:basic-bound}.6), we find
$$0\rightarrow M_\ell^{\pi_1(s)}/M_\ell^{\pi_1(S)}\rightarrow B^{\pi_1(s)}[\ell^\infty]/B^{\pi_1(S)}[\ell^\infty]\rightarrow Q\rightarrow 0,$$
where $Q$ is a subquotient of $\SH^3(X_{\overline{\eta}},\Z_\ell(1))[\ell^\infty]$ and $M_\ell^{\pi_1(s)}/M_\ell^{\pi_1(S)}\hookrightarrow M_\ell^U/M_\ell^{\pi_1(S)}$ by definition of $S_U$. In particular, (2) follows from Lemma \ref{lemme:formal} below. \end{proof}

\noindent  Let $T $ be a free $\Z_\ell$-module of rank $r$ and let $\Pi$ be a closed subgroup of $\SGL(T)$. Write again $V=T\otimes\Q_\ell$ and $M=T\otimes\Q_\ell/\Z_\ell\simeq V/T$. For any positive integer $n$, we can identify $M[\ell^n]$ with $T\otimes\Z/\ell^n\Z.$ Let $G_\ell$ denote the Zariski closure of $\Pi$ in $\SGL(V)$.

 \begin{lemme}\label{lemme:formal} Assume that $G$ is connected\footnote{This assumption is necessary, as can be shown by considering the $r=1$ case with $\Pi$ acting through a finite quotient and $U$ being trivial.}. Then, for every normal open subgroup $U\subset \Pi$,  the index of $M^\Pi$ in $M^U$ is finite. 
 \end{lemme}
\begin{proof}  Since $G$ is connected, we have $V^U=V^\Pi$ and $T^U=T\cap V^U=T\cap V^\Pi=T^\Pi$. The commutative diagram with exact lines
 $$\xymatrix{0\ar[r]&V^\Pi/T^\Pi\ar@{=}[d]\ar[r]& M^\Pi\ar@{_{(}->}[d]\ar[r]^{\delta_\Pi}&\SH^1(\Pi,T)\ar[d]^{res}\\
0\ar[r]&V^U/T^U\ar[r]& M^U\ar[r]^{\delta_U}&\SH^1(U,T)}$$ 
yields an isomorphism
 $$M^U/M^\Pi\simeq \Sim(\delta_U)/\Sim(\delta_\Pi\circ res).$$
In particular, the group $M^U/M^\Pi$ is a quotient of $\Sim(\delta_U)$ and the conclusion follows from the fact $\Sim(\delta_U)$ is finite as a torsion submodule of $\SH^1(U,T)$.
 \end{proof}

\subsection{}\hspace{-0.3cm} For every positive integer $d$, define 
$$S_{\leq d}=\bigcup_{[K:k]\leq d} S(K).$$
Write $S^{gen}_{\ell, \leq d}$ and $S^{ex}_{\ell, \leq d}$ accordingly. The following key result is a reformulation of \cite[Theorem 1.1 and Remark 3.1.2(1)]{UOI2}. 

\begin{fait}\label{proposition:uniform-bound} Assume that $S$ is a curve. Then for every positive integer $d$, the image of $S^{ex}_{\ell, \leq d}$ in $S$ is finite and there exists an open subgroup $U_d$ of $\rho_\ell(\pi_1(S))$ such that $S^{gen}_{\ell, \leq d}\subset S_{U_d}$.
\end{fait}

\bigskip

\begin{proof}[Proof of Theorem \ref{theorem:main}]
Without loss of generality, we may replace $k$ by a finite field extension and $S$ by its reduced subscheme. Also, from Proposition \ref{proposition:basic-bound}, we may replace $S$ by an open subscheme. Working separately on each connected component, we may assume that $S$ is smooth, separated and geometrically connected variety over $k$. Let  $U_d$ be an open subgroup of $\rho_\ell(\pi_1(S))$ as in Fact \ref{proposition:uniform-bound}. Let $s\in S_{\leq d}$ such that $X_s$ satisfies the $\ell$-adic Tate conjecture for divisors. 

If $s\in S_{\ell}^{gen}$, then  Fact \ref{proposition:uniform-bound} and (1) in Proposition \ref{proposition:bound-fixed-U} yield $|Br(X_{\overline s})^{\pi_1(s)}[\ell^\infty]|\leq C_{U_d}$, for some constant $C_{U_d}$ depending only on $U_d$.

If $s\in S_{\ell}^{ex}$ then, by Fact \ref{proposition:uniform-bound}, there are only finitely many possibilities for the image $s_0$ of $s$ in  $S$.  Since $\rho_\ell(\pi_1(s_0))$ is an $\ell$-adic Lie group, it has only finitely many open subgroups of index at most $d$. The intersection of all these open subgroups is an open subgroup $U_d(s_0)$ contained in $\rho_\ell(\pi_1(s))$. By construction $$Br(X_{\overline s})^{\pi_1(s)}[\ell^\infty]\subset Br(X_{\overline s})^{U_d(s_0)}[\ell^\infty].$$
From 
Proposition \ref{proposition:basic-bound} applied to $X_{s_0}$, $Br(X_{\overline s})^{U_d(s_0)}[\ell^\infty]$ is finite.
\end{proof}

\subsection{}\hspace{-0.3cm}Using Proposition \ref{proposition:bound-fixed-U}, (2) instead of (1)  in the proof above, one obtains the following unconditional variant of Theorem  \ref{theorem:main}.
 
\begin{theoreme}\label{theorem:mainbis} Let $k$ be a finitely generated field of characteristic zero and let $S$ be a quasi-projective curve over $k$. Let $\pi : X\ra S$ be a smooth, proper morphism. Then for every prime $\ell$ such that the Zariski-closure of the image of $\pi_1(S)$ acting on $\SH^2(X_{\overline{\eta}},\Q_\ell(1))$ is connected and for every positive integer $d$,  there exists a constant $C_d$  satisfying the following property: for every $s\in S_{\ell, \leq d}^{gen}$  $$[Br(X_{\overline s})^{\pi_1(s)}[\ell^\infty] : Br(X_{\overline \eta})^{\pi_1(S)}[\ell^\infty]]\leq C_d.$$
\end{theoreme}

 \noindent \begin{tabular}[t]{l}
\textit{anna.cadoret@imj-prg.fr}\\
IMJ-PRG -- Sorbonne Universit\'es, UPMC University Paris 6,\\
75005 PARIS, FRANCE\\
\\
\textit{francois.charles@math.u-psud.fr}\\
Universit\'e Paris-Sud,\\
91405 ORSAY, FRANCE.\\

\end{tabular}\\
  \end{document}